\documentclass[reqno,12pt]{amsart}

\usepackage{amsmath, amsthm, amsfonts, amssymb}
\input{mathrsfs.sty}

\usepackage{amsmath}
\usepackage{amsfonts}
\usepackage{amssymb}
\usepackage{eufrak}
\usepackage{amscd}
\usepackage{amsthm}
\usepackage{amstext}
\usepackage[all]{xy}

\newcommand{\id}{\operatorname{id}}

   \theoremstyle{plain}
   \newtheorem{thm}{Theorem}[section]
   \newtheorem{prop}[thm]{Proposition}
   \newtheorem{lem}[thm]{Lemma}
   \newtheorem{cor}[thm]{Corollary}
   \theoremstyle{definition}
   
   \newtheorem{defn}[thm]{Definition}
   \newtheorem{example}[thm]{Example}
   \theoremstyle{remark}

 \numberwithin{equation}{section}

\author{V. Manuilov}

\date{}

\address{Moscow Center for Fundamental and Applied Mathematics {\rm and} Moscow State University,
Leninskie Gory 1, Moscow, 
119991, Russia}

\email{manuilov@mech.math.msu.su}

\thanks{The author acknowledges partial support by the RFBR grant No. 18-01-00398}

\title{Metrics on doubles as an inverse semigroup}

\begin{document}

\begin{abstract}
For a metric space $X$ we study metrics on the two copies of $X$. We define composition of such metrics and show that the equivalence classes of metrics are a semigroup $M(X)$. Our main result is that $M(X)$ is an inverse semigroup. Therefore, one can define the $C^*$-algebra of this inverse semigroup, which is not necessarily commutative. If the Gromov--Hausdorff distance between two metric spaces, $X$ and $Y$, is finite then their inverse semigroups $M(X)$ and $M(Y)$ (and hence their $C^*$-algebras) are isomorphic. We characterize the metrics that are idempotents, and give examples of metric spaces for which the semigroup $M(X)$ (and the corresponding $C^*$-algebra) is commutative. We also describe the class of metrics determined by subsets of $X$ in terms of the closures of the subsets in the Higson corona of $X$ and the class of invertible metrics.

\end{abstract}

\maketitle

\section*{Introduction}

Given metric spaces $X$ and $Y$, a metric $d$ on $X\sqcup Y$ that extends the metrics on $X$ and $Y$ depends only on the values of $d(x,y)$, $x\in X$, $y\in Y$, but it may be hard to check which functions $d:X\times Y\to (0,\infty)$ determine a metric on $X\sqcup Y$: one has to check the triangle inequality too many times. The problem of description of all such extended metrics is difficult due to the lack of a nice algebraic structure on the set of metrics. 
It was a surprise for us to discover that in the case $Y=X$, there is a nice algebraic structure on the set $M(X)$ of quasi-isometry classes of extended metrics on the double $X\sqcup X$: it is an inverse semigroup. 

Recall that a semigroup $S$ is an inverse semigroup if for any $u\in S$ there exists a unique $v\in S$ such that $u=uvu$ and $v=vuv$ \cite{Lawson}. Philosophically, inverse semigroups describe local symmetries in a similar way as groups describe global symmetries, and technically, the construction of the (reduced) group $C^*$-algebra of a group generalizes to that of the (reduced) inverse semigroup $C^*$-algebra \cite{Paterson}. 

Thus, one can associate a new (noncommutative) $C^*$-algebra to any metric space. In particular, all quasi-isometry classes of metrics on the double of $X$ are partial isometries. We characterize the metrics that are idempotents in $M(X)$ and show that any two idempotents commute (which proves that $M(X)$ is an inverse semigroup). We show that if the Gromov--Hausdorff distance between two metric spaces, $X$ and $Y$, is finite then their inverse semigroups $M(X)$ and $M(Y)$ (and hence the corresponding $C^*$-algebras) are isomorphic. We also describe the class of metrics determined by subsets of $X$ in terms of the closures of the subsets in the Higson corona of $X$ and the class of invertible metrics, and give examples of metric spaces for which the semigroup $M(X)$ is commutative. 

\bigskip

Let $X=(X,d_X)$ be a metric space. 

\begin{defn}
A {\it double} of $X$ is a metric space $X\times\{0,1\}$ with a metric $d$ such that 
\begin{itemize}
\item
the restriction of $d$ on each copy of $X$ in
$X\times\{0,1\}$ equals $d_X$; 
\item
the distance between the two copies of $X$ is non-zero.
\end{itemize}

Let $\mathcal M(X)$ denote the set of all such metrics.

\end{defn}

We identify $X$ with $X\times\{0\}$, and write $X'$ for $X\times\{1\}$. Similarly, we write 
$x$ for $(x,0)$ and $x'$ for $(x,1)$, $x\in X$. 
Note that metrics on a double of $X$ may differ only when two points lie in different copies of $X$.
To define a metric $d$ in $\mathcal M(X)$ it suffices to define $d(x,y')$ for all $x,y\in X$. 

Recall that two metrics, $d_1$, $d_2$, on the double of $X$ are quasi-isometric if there exist $\alpha>0,\beta\geq 1$ such that 
$$
-\alpha+\frac{1}{\beta} d_1(x,y')\leq d_2(x,y')\leq \alpha+\beta d_1(x,y')
$$
for any $x,y\in X$. 
We call two metrics, $d_1$ and $d_2$, on the double of $X$ equivalent if they are quasi-isometric. In this case we write
$d_1\sim d_2$, or $[d_1]=[d_2]$.

\section{Composition of metrics} 

The idea to consider metrics on the disjoint union of two spaces as morphisms from one space to another was suggested in \cite{Manuilov-Morphisms}.

\begin{lem}\label{Lemma_d}
Let $(X,d_X)$, $(Y,d_Y)$ and $(Z,d_Z)$ be metric spaces, let $d$ be a metric on $X\sqcup Y$, $\rho$ a metric on $Y\sqcup Z$ such that $d|_X=d_X$, $d|_Y=\rho|_Y=d_Y$, $\rho|_Z=d_Z$. Then the formula
$$
b(x,z)=\inf_{y\in Y}[d(x,y)+\rho(y,z)],\quad x\in X,z\in Z,
$$ 
defines a metric on $X\sqcup Z$.

\end{lem}
\begin{proof}
Due to symmetry, it suffices to check the triangle inequality for the triangle $(x_1,x_2,z)$, $x_1,x_2\in X$, $z\in Z$. Fix $\varepsilon>0$ and let 
$y_1,y_2\in Y$ satisfy 
$$
d(x_1,y_1)+\rho(y_1,z)-b(x_1,z)<\varepsilon;\quad d(x_2,y_2)+\rho(y_2,z)-b(x_2,z)<\varepsilon.
$$ 
Then 
\begin{eqnarray*}
d_X(x_1,x_2)&\leq& d(x_1,y_1)+\rho(y_1,z)+\rho(z,y_2)+d(y_2,x_2)\leq b(x_1,z)+b(x_2,z)+2\varepsilon;\\
b(x_2,z)&\leq&d(x_2,y_1)+\rho(y_1,z)\leq d_X(x_2,x_1)+d(x_1,y_1)+\rho(y_1,z)\\
&\leq&d_X(x_2,x_1)+b(x_1,z)+\varepsilon.
\end{eqnarray*}
Taking $\varepsilon$ arbitrarily small, we obtain the triangle inequality.

\end{proof}

We shall denote the metric $b$ by $\rho\circ d$, or $\rho d$.

\begin{cor}
Let $\rho,d$ be metrics on the double of $X$. Then the formula
$$
\rho d(x,z')=\inf_{y\in X}[d(x,y')+\rho(y,z')],\quad x,z\in X,
$$ 
defines the composition of $d$ and $\rho$ on the double of $X$.

\end{cor}

\begin{lem}
The composition of metrics is associative.

\end{lem}
\begin{proof}
Obvious.

\end{proof}

\begin{lem}
If $d_1\sim \tilde d_1$ and $d_2\sim \tilde d_2$ then $\tilde d_1\circ \tilde d_2\sim d_1\circ d_2$.

\end{lem}
\begin{proof}
Suppose that there exist $\alpha\geq 0$, $\beta\geq 1$ such that $\tilde d_1(x,y')\leq\alpha+\beta d_1(x,y')$ and $\tilde d_2(x,y')\leq\alpha+\beta d_2(x,y')$ for any $x,y\in X$. 

Then 
\begin{eqnarray*}
\tilde d_1\circ \tilde d_2(x,z')&=&\inf_{y\in X}[\tilde d_2(x,y')+\tilde d_1(y,z')]\leq
\inf_{y\in X}[\alpha+\beta d_2(x,y')+\alpha+\beta d_1(y,z')]\\
&\leq&\inf_{y\in X}[2\alpha+\beta(d_2(x,y')+d_1(y,z'))]\leq 2\alpha+\beta d_1\circ d_2(x,z').  
\end{eqnarray*}
Lower bound is similar.

\end{proof}

Thus, the multiplication is well defined on equivalence classes of metrics on the double of $X$.

Denote the set of all equivalence classes of metrics on the double of $X$ by $M(X)=\mathcal M(X)/\sim$. Then $M(X)$ is a semigroup.


\begin{example}
If $X$ is discrete of finite diameter then all metrics on the double of $X$ are equivalent, so $M(X)$ consists of a single element.

\end{example}

\begin{example}\label{I}
Define a metric $I$ on the double of $X$ by $I(x,y')=d_X(x,y)+1$. (The triangle inequality obviously holds.) Note that $I\circ d\sim d\sim d\circ I$ for any metric on the double of $X$, hence 
$[I]$ is the unit element in the semigroup $M(X)$.

\end{example}

For a metric $d$ on the double of $X$ define the adjoint metric in $\mathcal M(X)$ $d^*$ by $d^*(x,y')=d(y,x')$, $x,y\in X$. Then $\ast$ is an involution: $(d^*)^*=d$ and $(d_1\circ d_2)^*=d_2^*\circ d_1^*$,
and it passes to the equivalence classes, making $M(X)$ a semigroup with involution.

A metric $d$ on the double of $X$ is selfadjoint if $d^*\in[d]$. Note that if $d$ is selfadjoint then there exists a metric $\tilde d\in[d]$ such that
$\tilde d^*=\tilde d$. Indeed, we can set $\tilde d(x,y')=\frac{1}{2}(d(x,y')+d(y,x'))$, $x,y\in X$.

The following simple statement from \cite{Manuilov-Morphisms} is the key observation allowing to see metrics as partial isometries.

\begin{prop}\label{inv-semi}
The metrics $d$ and $d\circ d^*\circ d$ are equivalent for any metric $d$ on the double of $X$.

\end{prop}
\begin{proof}
Let $x,y\in X$.  
On the one hand, taking $t=y$, $s=x$, we get
$$
(d\circ d^*\circ d)(x,y')=\inf_{t,s\in X}[d(x,t')+d^*(t,s')+d(s,y')]\leq 3d(x,y').
$$
On the other hand, passing to infimum in the triangle inequality, we get
\begin{eqnarray*}
(d\circ d^*\circ d)(x,y')&=&\inf_{t,s\in X}[d(x,t')+d^*(t,s')+d(s,y')]\geq \inf_{t,s\in X}[d(x,t')+d(t',s)+d(s,y')]\\
&\geq& d(x,y').
\end{eqnarray*}

\end{proof}

\begin{cor}
$[d^*d]$ is a selfadjoint idempotent for any metric $d$ on the double of $X$. 

\end{cor}

Recall that a semigroup $S$ is regular if for any $d\in S$ there is $b\in S$ such that $d=dbd$ and $b=bdb$.

\begin{cor}
$M(X)$ is a regular semigroup.

\end{cor}
\begin{proof}
Take $b=d^*$.

\end{proof}

There are typically a lot of idempotent metrics, i.e., metrics representing idempotents in $M(X)$ (see Example \ref{idem}) below. This means that, in general, $M(X)$ is not a cancellative semigroup. Indeed, if $d^2\sim d$ then cancellation would imply $d\sim I$.

\begin{example}\label{idem}
Let $X=\mathbb Z$ with the standard metric $d(n,m)=|n-m|$, $n,m\in\mathbb Z$. Set 
$$
d(n,m')=\left\lbrace\begin{array}{cl}n+m+1,&\mbox{if\ }n,m\geq 0;\\|n-m|+1,&\mbox{otherwise.}\end{array}\right.
$$ 
Then it is easy to see that $d^*=d$ and $[d\circ d]=[d]$, while $d$ is not quasi-isometric to $I$.

\end{example}

\section{Idempotents}

Denote by $d(x,X')$ the distance from $x\in X$ in the first copy of $X$ to the second copy $X'$ of $X$ in the double of $X$.

\begin{thm}\label{projections}
Let $d^*=d$ be a metric on the double of $X$. Then $[d^2]=[d]$ if and only if there exist $\alpha\geq 0$, $\beta\geq 1$ such that 
$-\alpha+\frac{1}{\beta} d(x,x')\leq d(x,X')$ for any $x\in X$.

\end{thm}
\begin{proof}
First, suppose that $[d^2]=[d]$. Then there exist $\alpha\geq 0$, $\beta\geq 1$ such that
$$
d^2(x,x')\geq -\alpha+\frac{1}{\beta} d(x,x').
$$
On the other hand,
$$
d^2(x,x')=\inf_{y\in X}[d(x,y')+d(y,x')]=\inf_{y\in X}2d(x,y')\leq 2d(x,X'),
$$
hence
$d(x,X')\geq -\frac{\alpha}{2}+\frac{1}{2\beta}d(x,x')$.

Second, suppose that there exist $\alpha\geq 0$, $\beta\geq 1$ such that
$d(x,x')\leq \alpha+\beta d(x,X')$ for any $x\in X$. We need to estimate $d^2(x,z')$ both from below and from above. The estimate 
from above is given by
\begin{eqnarray*}
d^2(x,z')&=&\inf_{y\in X}[d(x,y')+d(y,z')]\leq d(x,x')+d(x,z')\leq \alpha+\beta(d(x,X'))+d(x,z')\\
&\leq&\alpha+\beta d(x,z')+d(x,z')=\alpha+(\beta+1)d(x,z').
\end{eqnarray*}
Here we took $y=x$ and used that $d(x,X')\leq d(x,z')$ for any $z\in X$. 

To obtain an estimate from below, note that
\begin{equation}\label{d1}
d^2(x,z')=\inf_{y\in X}[d(x,y')+d(y,z')]=\inf_{y\in X}[d(x,y')+d(y',z)]\geq d_X(x,z).
\end{equation}

We also have 
\begin{equation}\label{d2}
d^2(x,z')\geq d(x,X')+d(z,X')\geq -\alpha+\frac{1}{\beta} d(x,x')-\alpha+\frac{1}{\beta} d(z,z'),
\end{equation}

It follows from (\ref{d1}) and (\ref{d2}) that

\begin{eqnarray}\label{d4}
d^2(x,z')&\geq& \frac{1}{2}d_X(x,z)-\alpha+\frac{1}{2\beta}(d(x,x')+d(z,z'))\nonumber\\
&\geq& -\alpha+\frac{1}{2\beta}(d_X(x,z)+d(x,x')+d(z,z')).
\end{eqnarray}

On the other hand, the triangle inequality shows that
\begin{equation}\label{d3}
d(x,z')\leq d(x,x')+d_X(x',z')+d(z',z)= d(x,x')+d_X(x,z)+d(z,z').
\end{equation}

Then (\ref{d4}) and (\ref{d3}) give 
$$
d^2(x,z')\geq -\alpha+\frac{1}{2\beta}d(x,z').
$$

\end{proof}


The next result shows that selfadjoint idempotents can be characterized only by the values $d(x,x')$, $x\in X$.

We call two functions, $\varphi,\psi:X\to [0,\infty)$, equivalent 
if there exist $\alpha\geq 0$, $\beta\geq 1$ such that $-\alpha+\frac{1}{\beta} \psi(x)\leq \varphi(x)\leq \alpha+\beta \psi(x)$ 
for any $x\in X$.

\begin{prop}\label{xx'}
Let $d$, $\rho$ be two idempotent metrics on the double of $X$, $\rho^*=\rho$, $d^*=d$. Then $\rho\sim d$ if and only if
the functions $x\mapsto \rho(x,x')$ and $x\mapsto d(x,x')$ are equivalent.

\end{prop}
\begin{proof}
One direction is trivial, so let us prove the non-trivial one.

Since $d$ is a selfadjoint idempotent, there are $\alpha\geq 0$, $\beta\geq 1$ such that
\begin{equation}\label{xx-1}
d(x,x')\leq \alpha+\beta d(x,X')\leq\alpha+\beta d(x,z')
\end{equation}
for any $x,z\in X$.

If the functions $x\mapsto \rho(x,x')$ and $x\mapsto d(x,x')$ are equivalent then there exist $\gamma>0$, $\delta\geq 1$ such that
\begin{equation}\label{xx-2}
\rho(x,x')\leq \gamma+\delta d(x,x')\leq \gamma+\delta(\alpha+\beta d(x,x'))=\alpha'+\beta' d(x,z'),.
\end{equation}
where $\alpha'=\gamma+\delta\alpha$, $\beta'=\delta\beta$.

Using (\ref{xx-1}) and the triangle inequality, we have
\begin{equation}\label{xx-3}
d_X(x,z)=d(x',z')\leq d(x',x)+d(x,z')\leq \alpha+(1+\beta)d(x,z').
\end{equation}

Using the triangle inequality again, together with (\ref{xx-2}) and (\ref{xx-3}), we have
\begin{eqnarray*}
\rho(x,z')&\leq&\rho(x,x')+\rho(x',z')=\rho(x,x')+d_X(x,z)\\
&\leq& \alpha'+\beta'd(x,z')+\alpha+(1+\beta)d(x,z')=(\alpha+\alpha')+(1+\beta+\beta')d(x,z'),
\end{eqnarray*}
i.e. $d$ dominates $\rho$. Symmetrically, $\rho$ dominates $d$, hence they are equivalent.

\end{proof}

It would be interesting to find a characterization of functions on $X$ which can be obtained from metrics on doubles.
The next statement shows that selfadjoint idempotents in $M(X)$ commute.

\begin{prop}\label{Prop23}
Let $d$, $\rho$ be two idempotent metrics on the double of $X$, $\rho^*=\rho$, $d^*=d$. Then $d\rho\sim\rho d$.

\end{prop}
\begin{proof}

By definition, for any $\varepsilon>0$ there exists $y_0\in X$ such that
\begin{equation}\label{y0}
d\rho(x,z')=\inf_{y\in X}[\rho(x,y')+d(y,z')]\geq \rho(x,y'_0)+d(y_0,z')-\varepsilon.
\end{equation}
By Theorem \ref{projections}, there exist $\alpha\geq 0$, $\beta\geq 1$ such that
$$
\rho(x,y'_0)\geq-\alpha+\frac{1}{\beta}\rho(y_0,y'_0);\quad d(y_0,z')\geq-\alpha+\frac{1}{\beta} d(y_0,y'_0).
$$
Then
\begin{equation}\label{y1}
d\rho(x,z')+\varepsilon\geq \frac{1}{\beta}(\rho(y_0,y'_0)+d(y_0,y'_0))-2\alpha.
\end{equation}

The triangle inequality applied to the right hand side of (\ref{y0}) gives
\begin{equation}\label{y2}
d\rho(x,z')+\varepsilon\geq \rho(x,y_0)-\rho(y_0,y'_0)+d(z,y_0)-d(y_0,y'_0). 
\end{equation}

On the other hand,
\begin{equation}\label{y3}
\rho d(x,z')\leq d(x,y'_0)+\rho(z,y'_0)\leq d(x,y_0)+d(y_0,y'_0)+\rho(y_0,y'_0)+\rho(y_0,z).
\end{equation}

Denote $d(x,y_0)+\rho(y_0,z)=d_X(x,y_0)+d_X(y_0,z)$ by $r$ and $d(y_0,y'_0)+\rho(y_0,y'_0)$ by $s$.
Then (\ref{y1}) and (\ref{y2}) can be written as
$$
d\rho(x,z')\geq \max\Bigl(\frac{1}{\beta} s-2\alpha,r-s\Bigr)-\varepsilon,
$$
and (\ref{y3}) can be written as
$$
\rho d(x,z')\leq r+s. 
$$

To finish the argument, we need the following statement.
\begin{lem}\label{ab} 
There exists $\lambda>1$ such that $r+s\leq \lambda (\max(\frac{1}{\beta}s-2\alpha,r-s)+2\alpha)$ for any $r,s\geq 0$.

\end{lem}
\begin{proof}
First, note that $\max(\frac{1}{\beta}s,r-s)\leq \max(\frac{1}{\beta}s-2\alpha,r-s)+2\alpha$. It remains to show that 
\begin{equation}\label{ab1}
r+s\leq\lambda \max\Bigl(\frac{1}{\beta}s,r-s\Bigr) 
\end{equation}
for some $\lambda>1$. Set $s=t\,r$, $t\in[0,\infty)$.
Then (\ref{ab1}) becomes
$$
(1+t)r\leq\lambda\max\Bigl(\frac{t}{\beta}r,(1-t)r\Bigr),
$$
or, simply,
\begin{equation}\label{ab2}
(1+t)\leq\lambda\max\Bigl(\frac{t}{\beta},(1-t)\Bigr)
\end{equation}
Taking $\lambda=\beta(2+\beta)$, we can provide that (\ref{ab2}) holds for any $t\in[0,\infty)$.

\end{proof}
Lemma \ref{ab} implies that $\rho d(x,z')\leq \lambda d\rho(x,z')+2\alpha+\varepsilon$. Symmetry implies that $\rho d$ and $d\rho$ are equivalent. This finishes the proof of Proposition \ref{Prop23}.

\end{proof}

\section{Metrics from subsets}

\begin{example}
Let $A\subset X$ be a subset. Define a metric $d_A$ on the double of $X$ by
$$
d_A(x,y')=\inf_{z\in A}[d_X(x,z)+1+d_X(z,y)].
$$
Then $d_A$ is selfadjoint, $d_A(x,x')=\inf_{z\in A}[2d(x,z)+1]=2d(x,A)+1$, and
\begin{eqnarray*}
d_A(x,X')&=&\inf_{y\in X,z\in A}[d_X(x,z)+1+d_X(z,y)]=\inf_{z\in A}[d_X(x,z)+1]\\
&=&d(x,A)+1, 
\end{eqnarray*}
hence $[d_A]$ is an idempotent.

\end{example}

A special case of the above metrics $d_A$, $A\subset X$, is the case $A=\{x_0\}$ for a fixed point $x_0\in X$.
It is clear that the equivalence class of the metric $e_0=d_{\{x_0\}}$ does not depend on the choice of the point $x_0$.

Let $A,B\subset X$ be closed subsets. In this section we establish when $[d_A]=[d_B]$ under the assumption that $X$ is locally compact (and the metric $d_X$ is proper).

Recall that the Higson compactification $h X$ of a locally compact metric space $X$ is the Gelfand dual of the $C^*$-algebra $C_h(X)$ of bounded continuous functions $f$ on $X$ such that $\lim_{x\to\infty}\operatorname{Var}_r(f)(x)=0$ for any $r>0$, where 
$$
\operatorname{Var}_r(f)(x)=\sup_{y\in X,d_X(x,y)\leq r}|f(x)-f(y)|. 
$$
The Gelfand dual of the quotient $C_h(X)/C_0(X)$ is the Higson corona $\nu X=h X\setminus X$.  

Let $J_A=\{f\in C_h(X):f|_A=0\}$. This is an ideal in $C_h(X)$. Then the Gelfand dual of $C_h(X)/J_A$ is the closure $\bar{A}$ of $A$ in $h X$, and $\bar{A}\setminus A=\bar{B}\setminus B$ if and only if $J_A+C_0(X)=J_B+C_0(X)$.

\begin{prop}\label{AB}
The following are equivalent:
\begin{enumerate}
\item
$[d_A]=[d_B]$; 
\item
there exists $C>0$ such that $A$ lies in the $C$-neighborhood of $B$ and $B$ lies in the $C$-neighborhood of $A$;
\item
$\bar{A}\setminus A=\bar{B}\setminus B$ in the Higson corona.
\end{enumerate}

\end{prop}
\begin{proof}
We begin with (1)$\iff$(2). 
Suppose first that $[d_A]=[d_B]$. Note that $d_A(x,x')=2d_X(x,A)+1$. In particular, $d_A(x,x')=1$ when $x\in A$. Then there exists $C>0$ such that $d_B(x,x')=2d_B(x,B)<C$ for any $x\in A$, in other words, $A$ lies in the $C$-neighborhood of $B$. Similarly, $B$ lies in the $C$-neighborhood of $A$ (maybe with another $C$). 

Assume now that (2) holds. Then $d_X(x,A)-C\leq d_X(x,B)\leq d_X(x,A)+C$, hence the functions $d_A(x,x')=2d_X(x,A)+1$ and $d_B(x,x')=2d_X(x,B)+1$ are equivalent. By Proposition \ref{xx'} we are done.

Now let us show that (2)$\iff$(3). Let (2) hold, and let 
$f\in J_A$. Let $x_0\in X$. Let $B\subset N_C(A)$. Let $r:X\to [0,\infty)$ and $\mu:X\to [0,\infty)$ be defined by $r(x)=d_X(x,x_0)$ and by $\mu(x)=d_X(x,A)$ respectively. Define the map $\gamma:X\to[0,\infty)\times[0,\infty)$ by $\gamma(x)=(r(x),\mu(x))$. Let 
$$
F_0=\gamma^{-1}([0,\infty)\times[0,C]); \  F_1=\gamma^{-1}([0,\infty)\times[2C,\infty)); \ D_k=\gamma^{-1}([k-1,k]\times[C,2C]).
$$ 
Then $X=F_0\cup F_1\cup(\bigcup_{k=1}^\infty D_k)$.  

For the function $f\in J_A$, set $f_{n}=\sup\{|f(x)|:x\in \bigcup_{k=n+1}^\infty D_k\}$. As $\bigcup_{k=1}^\infty D_k\subset N_{2C}(A)$ and as $f\in C_h(X)$ satisfies $f|_A=0$, one has $\lim_{n\to\infty}f_n=0$.

Let us construct a function $g\in C_h(X)$. Set $g|_{F_0}=0$ and $g|_{F_1}=f|_{F_1}$. Our aim is to extend $g$ to the whole $X$ with the following properties: 

\begin{eqnarray}\label{estim}
\|g|_{D_n}\|&\leq& 2f_{n-1};\\ 
\|g|_{E_n}\|&\leq& 2f_n, \nonumber 
\end{eqnarray}
where $E_n=\gamma^{-1}(\{n\}\times[C,2C])$. We construct such $g$ inductively. Suppose that we have already extended $g$ to $F_0\cup F_1\cup(\bigcup_{k=1}^{n}D_k)$. By the Tietze Extension Theorem, extend $g$ to $D_{n+1}$, and denote this extension on $D_{n+1}$ by $\tilde g$. As $\|g|_{E_n}\|\leq 2f_n$ and $|f(x)|\leq f_n$ for any $x\in\gamma^{-1}([n,n+1]\times\{2C\})$, we have $\|\tilde g|_{D_{n+1}}\|\leq 2f_n$. 

As $|g(x)|\leq f_{n+1}$ for any $x\in\gamma^{-1}(\{n+1\}\times\{2C\})$, there exists $C_0\in[C,2C]$ such that $|\tilde g(x)|\leq 2 f_{n+1}$ for any $x\in\gamma^{-1}(\{n+1\}\times[C_0,2C])$. Let $\varphi:[n,n+1]\times[C,2C]\to[0,1]$ be a continuous function such that $\varphi(r,2C)=\varphi(n,\mu)=1$ for any $r\in[n,n+1]$, $\mu\in[C,2C]$, and $\varphi(n+1,\mu)=0$ for $\mu\in[C,C_0]$. Then set $g(x)=\tilde g(x)\varphi(\gamma(x))$ for $x\in D_{n+1}$. Then $g$ is continuous on $F_0\cup F_1\cup D_1\cup\cdots\cup D_{n+1}$, $\|g|_{D_{n+1}}\|\leq 2f_n$ and $\|g|_{E_{n+1}}\|\leq 2f_{n+1}$. 

By construction, $g|_B=0$, and it follows from (\ref{estim}) that $f-g\in C_0(X)$, therefore $g\in J_B+C_0(X)$, i.e., $J_A\subset J_B+C_0(X)$. Symmetrically, $J_B\subset J_A+C_0(X)$.

Now, suppose that (3) holds, i.e., $J_A+C_0(X)=J_B+C_0(X)$. If $A$ does not lie in a $C$-neighborhood of $B$ for any $C$ then there exists a sequence $x_n\in A$, $n\in\mathbb N$, such that $\lim_{n\to\infty}d_X(x_n,B)=\infty$. Note that, necessarily, $\lim_{n\to\infty}x_n=\infty$.

Passing to a subsequence of $(x_n)_{n\in\mathbb N}$, it may be arranged that $d(x_n,B)\geq n$ and
$d(x_n,x_j)\geq n+j$ for all $n,j\in\mathbb N$. Let $h_n(x) = (1-d(x, x_n)/n)_+$ be the positive part of
$1 - d(x, x_n)/n$. This function is Lipschitz with constant $1/n$ and supported in the
ball of radius $n$ around $x_n$. By assumption, these balls are all disjoint and disjoint
from $B$ as well. The sum $f=\sum_{n\in\mathbb N}h_n$ belongs to $C_h(X)$. As $f|_B=0$, we have $f\in J_B$, hence $f\in J_A+C_0(X)$. On the other hand, as $f(x_n)=1$ for any $n\in\mathbb N$, hence $f\notin J_A+C_0(X)$ 
(recall that all $x_n$, $n\in\mathbb N$, lie in $A$, and $\lim_{n\to\infty}x_n=\infty$). This contradiction finishes the proof.

\end{proof} 


Note that an arbitrary idempotent metric need not be equivalent to $d_A$ for any $A$.

\section{Order structure}

Let $\rho$ and $d$ be metrics on the double of $X$. We say that $[\rho]\preceq[d]$ if there exists a metric $d'\in[d]$ such that
$d'(x,z')\leq \rho(x,z')$ for any $x,z\in X$. This gives a partial order on $M(X)$. 

\begin{lem}
Let $d$, $\rho$ be selfadjoint idempotent metrics on the double of $X$. Then $d\preceq \rho$ if and only if $[\rho][d]=[d]$.

\end{lem}
\begin{proof}
First, suppose that $[\rho d]=[d]$. Since both $\rho$ and $d$ are selfadjoint idempotents, there exist $\alpha\geq 0,\beta\geq 1$ such that
$\rho(x,X')\geq\frac{1}{\beta}\rho(x,x')-\alpha$ and $d(x,X')\geq\frac{1}{\beta}d(x,x')-\alpha$ for any $x\in X$.
Then
\begin{eqnarray*}
\rho\circ d(x,x')&\geq&\inf_{y\in X}[d(x,y')+\rho(y,x')]\geq d(x,X')+\rho(x,X')\geq \rho(x,X')\\
&\geq& \frac{1}{\beta}\rho(x,x')-\alpha.
\end{eqnarray*}
It follows from $[\rho\circ d]=[d]$ that there exist $\alpha'>0,\beta'>1$ such that $\rho\circ d(x,x')\leq \beta'd(x,x')+\alpha'$. 
Combining the last two inequalities,
we get $\beta'd(x,x')+\alpha'\geq \frac{1}{\beta}\rho(x,x')-\alpha$, or $d(x,x')\geq \frac{1}{\beta\beta'}\rho(x,x')-\alpha''$ for some $\alpha''>0$. 
Using the proof of Proposition \ref{xx'}, we
conclude that $d\preceq \rho$.

Second, suppose that $d\preceq \rho$. Without loss of generality, we may assume that $d(x,z')\geq\rho(x,z')$ for any $x,z\in X$. 
Then
$$
\rho\circ d(x,x')=\inf_{y\in X}[d(x,y')+\rho(y,x')]\leq d(x,x')+\rho(x,x')\leq 2d(x,x').
$$
This implies that $[d]\preceq[\rho d]$. On the other hand,
$$
\rho\circ d(x,x')\geq d(x,X')+\rho(x,X')\geq d(x,X')\geq \frac{1}{\beta}(d(x,x')-\alpha),
$$
which implies $[\rho d]\preceq[d]$. Thus, $[\rho d]=[d]$.

\end{proof}

\section{$C^*$-algebra of $M(X)$}


\begin{prop}
Let $a\in M(X)$ be an idempotent. Then it is selfadjoint.

\end{prop}
\begin{proof}
Note that $a^*$ also must be an idempotent. Then use commutativity of selfadjoint idempotents to show that
$$
a^*=a^*aa^*=(a^*a)(aa^*)=(aa^*)(a^*a)=aa^*a=a.
$$

\end{proof}

\begin{cor}
Any two idempotents in $M(X)$ commute.

\end{cor}

Recall that a semigroup $S$ is an inverse semigroup if for any $a\in S$ there exists a unique $b\in S$ such that $a=aba$ and $b=bab$ (\cite{Lawson}, p. 6).

\begin{thm}\label{T-inv}
$M(X)$ is an inverse semigroup.

\end{thm}
\begin{proof}
In a regular semigroup, commutativity of idempotents is equivalent to being an inverse semigroup (\cite{Lawson}, Theorem 3).

\end{proof}

By Theorem \ref{T-inv}, we can define the (reduced) semigroup $C^*$-algebra $C^*_r(M(X))$ of the inverse semigroup $M(X)$ (\cite{Paterson}, Section 4.4).
 
Recall that if $a\in M(X)$, $V_a=\{b\in M(X):bb^*\preceq a^*a\}$, then the map $b\mapsto ab$ is injective on $V_a$.

Let $l_2(M(X))$ denote the Hilbert space of square-summable functions on $M(X)$ (as a discrete space) 
with the orthonormal basis (often uncountable) of delta-functions $\delta_b(c)=\left\lbrace\begin{array}{cl}1,&\mbox{if\ }c=b;\\
0,&\mbox{if\ }c\neq b,\end{array}\right.$ $b,c\in M(X)$.
For $a\in M(X)$, set 
$$
\lambda_a(\delta_b)=\left\lbrace\begin{array}{cl}\delta_{ab},&\mbox{if\ }b\in V_a;\\0,&\mbox{if\ }b\notin V_a.\end{array}\right.
$$
Then $\lambda_a$ is a partial isometry for any $a\in M(X)$, and the $C^*$-algebra $C^*_r(M(X))$ generated by all $\lambda_a$,
$a\in M(X)$, is the reduced $C^*$-algebra of $M(X)$.

There is a special projection $[e_0]$ in $C^*_r(M(X))$, given by the metric $d_{\{x_0\}}$ for some $x_0\in X$. 
By definition, $e_0(x,y')=d_X(x,x_0)+1+d_X(x_0,y)$, and it is easy to see that
the equivalence class $[e_0]$ does not depend on $x_0$.

\begin{lem}\label{zero}
$d\circ e_0$ and $e_0\circ d$ are equivalent to $e_0$ for any metric $d$ on the double of $X$.

\end{lem}
\begin{proof}
Take $y=x_0$, and then use the triangle inequality to obtain
\begin{eqnarray*}
e_0d(x,z')&=&\inf_{y\in X}[d(x,y')+d_X(y,x_0)+d_X(z,x_0)+1]\leq d(x,x'_0)+d_X(z,x_0)+1\\
&\leq& d_X(x,x_0)+d(x_0,x'_0)+d_X(z,x_0)+1\leq e_0(x,z')+d(x_0,x'_0).
\end{eqnarray*}

On the other hand, by the triangle inequality,
\begin{eqnarray}\label{inf}
d(x,y')+d_X(y,x_0)+d_X(z,x_0)+1&=&d(x,y')+d_X(y',x'_0)+d_X(z,x_0)+1\nonumber\\
&\geq& d(x,x'_0)+d_X(z,x_0)+1. 
\end{eqnarray}
Passing to the infimum in (\ref{inf}) with respect to $y\in X$, we obtain $e_0d(x,z')\geq d(x,x'_0)+d_X(z,x_0)+1$. Another triangle inequality gives
\begin{eqnarray*}
e_0d(x,z')&\geq& d(x,x'_0)+d_X(z,x_0)+1\geq d_X(x,x_0)-d(x_0,x'_0)+d_X(z,x_0)+1\\&=&e_0(x,z')-d(x_0,x'_0).
\end{eqnarray*}
Thus, $e_0(x,z')-\alpha\leq e_0d(x,z')\leq e_0(x,z')+\alpha$ for any $x,z\in X$, where $\alpha=d(x_0,x'_0)$, hence $e_0d\sim e_0$. Similarly, $de_0\sim e_0$.

\end{proof}
Thus, $[e_0]$ is the zero element in $M(X)$.

\begin{prop}\label{minimalprojection}
The set $V_{e_0}$ consists of a single element $[e_0]$.

\end{prop} 
\begin{proof}
Let $s\in V_{[e_0]}$. Then $ss^*\preceq[e_0]$, hence $ss^*=ss^*[e_0]=[e_0]$. Then $s=ss^*s=[e_0]s=[e_0]$.

\end{proof}

\begin{cor}
$\lambda_{[e_0]}$ is a rank one projection in $C^*_r(M(X))$.

\end{cor}

\begin{cor}
There is a direct sum decomposition of $C^*$-algebras $C^*_r(M(X))=C^*_0(M(X))\oplus\lambda_{e_0}\mathbb C$, where $C^*_0(M(X))=
\{\lambda_a: a\in C^*_r(M(X)),a[e_0]=[e_0]a=0\}$.

\end{cor}

Let $A,B\subset X$, let $d_X(A,B)=\inf_{x\in A;y\in B}d_X(x,y)$, and let $B_R(x_0)$ denote the ball of radius $R$ centered at $x_0\in X$.

\begin{prop}\label{AB0}
Suppose that there exists $\beta\geq 1$ such that
\begin{equation}\label{*}
d_X(A\setminus B_R(x_0),B\setminus B_R(x_0))>\frac{1}{\beta} R.
\end{equation}
Then $[d_Ad_B]=[e_0]$.

\end{prop}
\begin{proof}
By Proposition \ref{xx'}, it suffices to compare $d_Ad_B(x,x')$ and $e_0(x,x')$, $x\in X$. Recall that
$$
d_Ad_B(x,x')=\inf_{u\in A,v\in B,y\in X}[d_X(x,u)+d_X(u,y)+d_X(y,v)+d_X(v,x)+2], 
$$
$$
e_0(x,x')=2d_X(x,x_0)+1.
$$
By Proposition \ref{minimalprojection}, $[e_0]\preceq[d_Ad_B]$, so it remains to show that $[d_Ad_B]\preceq[e_0]$.

Set $L=\frac{1}{\beta}R$. Take $x\in X$, and let $R$ satisfy $x\in B_{2R}(x_0)$ and $x\notin B_{R+L}(x_0)$. Then
\begin{equation}\label{AB1}
e_0(x,x_0)=2d(x,x_0)+1\leq 4R+1.
\end{equation}
Now let us estimate $d_Ad_B(x,x')$. By definition, there exist $u_0\in A$, $v_0\in B$, $y_0\in X$ such that
$$
d_Ad_B(x,x')\geq d_X(x,u_0)+d_X(u_0,y_0)+d_X(y_0,v_0)+d_X(v_0,x).
$$

Consider the two cases: 

(a) either $u_0\in B_R(x_0)$ or $v_0\in B_R(x_0)$;

(b) $u_0,v_0\notin B_R(x_0)$.

In the case (a), if $u_0\in B_R(x_0)$ and $x\notin B_{R+L}(x_0)$ then $d_X(x,u_0)\geq L$. Otherwise, if $v_0\in B_R(x_0)$ then $d_X(v_0,x)\geq L$. Thus, $d_Ad_B(x,x')\geq L$.

In the case (b), by the triangle inequality and by (\ref{*}),
\begin{eqnarray*}
d_Ad_B(x,x')&\geq& d_X(x,u_0)+d_X(u_0,y_0)+d_X(y_0,v_0)+d_X(v_0,x)\\
&\geq& d_X(x,u_0)+d_X(u_0,v_0)+d_X(v_0,x)
\geq d_X(u_0,v_0)\geq L.
\end{eqnarray*}
Thus, in both cases we have 
\begin{equation}\label{AB2}
d_Ad_B(x,x')\geq L=\frac{1}{\beta}R.
\end{equation}
Combining (\ref{AB1}) and (\ref{AB2}), we get $e_0(x,x')\leq 4\beta d_Ad_B(x,x')+1$, hence $[d_Ad_B]\preceq[e_0]$.

\end{proof}

\begin{cor}
Under the assumption of Proposition \ref{AB0}, $(\lambda_{[d_A]}-\lambda_{[e_0]})(\lambda_{[d_B]}-\lambda_{[e_0]})=0$, i.e., the projections $\lambda_{[d_A]}-\lambda_{[e_0]}$ and $\lambda_{[d_B]}-\lambda_{[e_0]}$ are mutually orthogonal.

\end{cor}

\begin{example}
Let $X=\mathbb R^2$ with the standard metric. For $\varphi\in[0,2\pi)$ let $A_\varphi$ be the ray from the origin with the angle $\varphi$ to the polar axis. If $\psi\in[0,2\pi)$, $\psi\neq\varphi$, then the two rays $A_\varphi$ and $A_\psi$ satisfy the assumption of Proposition \ref{AB0}, hence the projections $\lambda_{[d_{A_\varphi}]}-\lambda_{[e_0]}$ and $\lambda_{[d_{A_\psi}]}-\lambda_{[e_0]}$ are mutually orthogonal. Thus, the $C^*$-algebra $C^*(M(X))$ has uncountably many mutually orthogonal projections.

\end{example}

\section{Examples}




\begin{prop}
Let $X$ be a closed subset of $[0,\infty)$ with the induced metric. Then any $a\in M(X)$ is a selfadjoint idempotent. Hence $M(X)$ is commutative.

\end{prop}
\begin{proof}

First, let us show that any element of $M(X)$ is selfadjoint.
Suppose the contrary. Then there exists a metric $d$ on the double of $X$ such that $d^*$ is not equivalent to $d$, and for any $n\in\mathbb N$ we can find points $y_n,z_n\in X$ such that 
\begin{equation}\label{sparce1}
n\cdot d(y_n,z'_n)< d(y'_n,z_n).
\end{equation}
Since $d(X,X')>0$, the sequence $d(y'_n,z_n)$ is unbounded.  

Passing to a subsequence, if necessary, we may assume without loss of generality that $y_n<z_n$ for any $n\in\mathbb N$.
Then $d_X(x_0,z_n)=d_X(x_0,y_n)+d_X(y_n,z_n)$. 

By the triangle inequality, we have
\begin{equation}\label{sparce2}
d(y'_n,z_n)\leq 2d_X(y_n,z_n)+d(y_n,z'_n),
\end{equation}
so, (\ref{sparce1}) and (\ref{sparce2}) imply that
$$
n\cdot d(y_n,z'_n)< 2d_X(y_n,z_n)+d(y_n,z'_n),
$$
or, equivalently,
\begin{equation}\label{sparce3}
d(y_n,z'_n)<\frac{2}{n-1}d_X(y_n,z_n).
\end{equation}
Another application of the triangle inequality gives 
$$
d(y_n,z'_n)\geq d_X(x_0,z_n)-(d_X(y_n,x_0)+d(x_0,x'_0)).
$$ 
Combining this with (\ref{sparce3}), we get
\begin{equation}\label{sparce4}
d_X(x_0,z_n)-(d_X(y_n,x_0)+d(x_0,x'_0))<\frac{2}{n-1}d_X(y_n,z_n).
\end{equation}
By assumption, $d_X(x_0,z_n)=d_X(x_0,y_n)+d_X(y_n,z_n)$, so (\ref{sparce4}) implies that
$$
d_X(y_n,z_n)-d(x_0,x'_0)<\frac{2}{n-1}d_X(y_n,z_n)
$$
holds, hence the values $d_X(y_n,z_n)$ are uniformly bounded. Let $C$ satisfy $d_X(y_n,z_n)<C$ for any $n\in\mathbb N$.

By the triangle inequality and (\ref{sparce1}), we have
\begin{eqnarray*}
n(d(y_n,y'_n)-C)&<&n(d(y_n,y'_n)-d_X(y'_n,z'_n))\leq n d(y_n,z'_n)\\
&<&d(y'_n,z_n)\leq d(y_n,y'_n)+d_X(y_n,z_n)<d(y_n,y'_n)+C,
\end{eqnarray*}
hence $d(y_n,y'_n)<\frac{n+1}{n-1}C$, and the values $d(y_n,y'_n)$ are uniformly bounded.

Thus we get a contradiction: the left-hand side of the triangle inequality
$$
d(y'_n,z_n)\leq d(y_n,y'_n)+d_X(y_n,z_n)
$$
is unbounded, while both summands in the right-hand side are uniformly bounded.

Second, we have to show that any selfadjoint metric $d$ on the double of $X$ is an idempotent.
Suppose the contrary: for any $n\in\mathbb N$ there exist points $y_n,z_n\in X$ such that
\begin{equation}\label{s1}
d(y_n,z'_n)<\frac{1}{n}d(y_n,y'_n).
\end{equation}
Once again, we have two possibilities: either $d_X(x_0,z_n)=d_X(x_0,y_n)+d_X(y_n,z_n)$ or $d_X(x_0,y_n)=d_X(x_0,z_n)+d_X(z_n,y_n)$ for infinitely many numbers $n$'s, and let us assume that the first opportunity holds true. 

Then, by the triangle inequality, we have
$$
d_X(z_n,x_0)-(d_X(y_n,x_0)+d(x_0,x'_0))\leq d(y_n,z'_n),
$$
or, equivalently,
$$
d_X(y_n,z_n)-d(x_0,x'_0)\leq d(y_n,z'_n),
$$
which, together with (\ref{s1}), implies that
\begin{equation}\label{s2}
d_X(y_n,z_n)\leq d(y_n,z'_n)+d(x_0,x'_0)<\frac{1}{n}d(y_n,y'_n)+d(x_0,x'_0).
\end{equation}
Another triangle inequality combined with (\ref{s1}) and (\ref{s2}) gives
$$
d(y_n,y'_n)\leq d(y_n,z'_n)+d_X(y'_n,z'_n)=d(y_n,z'_n)+d_X(y_n,z_n)<\frac{2}{n}d(y_n,y'_n)+d(x_0,x'_0),
$$
which holds for infinitely many $n$'s. The latter may be true only if $d(y_n,y'_n)$ is bounded for these $n$'s, but this contradicts $d(X,X')>0$. Indeed, if $d(y_n,y'_n)<C$ for some $C>0$ and for infinitely many $n$'s then the sequence $d(y_n,z'_n)$ is not separated from 0. 

\end{proof}

\begin{example}
Let $X=\{(n,n,0):n\in\mathbb N\}\cup\{(n,-n,0):n\in\mathbb N\}\subset\mathbb R^3$ with the standard metric, and let
$X'=\{(x,-y,1):(x,y,0)\in X\}$. For $a_n=(n,n,0)\in X$ we have $a'_n=(n,-n,1)\subset X'$, and for
$b_n=(n,-n,0)\in X$ we have $b'_n=(n,n,1)\in X'$. 
Let the metric $d$ on the double of $X$ be inherited from the standard metric of $\mathbb R^3$.
Then $d(a_n,a'_n)=\sqrt{n^2+1}$, while $d(a_n,X')=d(a_n,b'_n)=1$, hence $[d]$ is selfadjoint, but not idempotent.  

\end{example}

\begin{example}
Let $X=\mathbb R$ with the standard metric, and let $A=[0,\infty)$, $B=(-\infty,0]$. For $x,y\in X$, set $d(x,y')=\left\lbrace\begin{array}{cl}|x+y|+1,&\mbox{if\ }x\in A,y\in B;\\|x|+|y|+1,&\mbox{otherwise.}\end{array}\right.$ 

Then $d^*d(x,y')=\inf_{z\in X}d(x,z')+d(y,z')$. If $x,y\in A$ then one may take $z=-x$, in this case $d(x,-x')+d(y,-x')=|x-y|+2$. In other cases one may take $z=0$, and $d^*d(x,y')=|x|+|y|+2$. Thus, $d^*d(x,y')=d_A(x)+1$, hence $[d^*d]=[d_A]$. Similary, we can see that $[dd^*]=[d_B]$. Thus, $[d]$ is a partial isometry from $[d_A]$ to $[d_B]$.

\end{example}

\section{Examples from extended metrics}

When $X$ is non-compact, the inverse semigroup $M(X)$ is infinite. Here we consider the case when metrics are replaced by the so-called extended metrics, which are the same as usual metrics, except that they are allowed to take infinite values. This gives a lot of examples with {\it finite} $M(X)$.

Note that setting $d(x,y')=\infty$ for any $x,y\in X$ gives the zero element $0\in M(X)$, as $d0=0d=0$ for any metric $d$ on the double of $X$.

\begin{example}
Let $X$ be a one-point space, $X=\{a\}$. Any two finite metrics on the double of $X$ are equivalent, but an infinite metric with $d(a,a')=\infty$ is not equivalent to a finite metric, so $M(X)=\{I,0\}$. We have $V_I=\{I,0\}$ and $V_0=\{0\}$. Then the $C^*$-algebra of $X$ is a subalgebra in the algebra $M_2(\mathbb C)$ of $2\times 2$ matrices, generated by the identity matrix and by a rank one projection, hence is isomorphic to $\mathbb C\oplus\mathbb C$. 

\end{example}

\begin{example}
Let $X$ be the space consisting of two points, $a$ and $b$, with $d_X(a,b)=\infty$. Any metric in $\mathcal M(X)$ is determined by the 4 values: $d(a,a')$, $d(a,b')$, $d(b,a')$ and $d(b,b')$. Metrics with any finite value are equivalent to those with this value equal to 1, so non-equivalent classes of metrics should take values 1 and $\infty$. Taking into account the triangle inequality, there are 7 possible metrics in $M(X)$:
\begin{enumerate}
\item
$0(a,a')=0(a,b')=0(b,a')=0(b,b')=\infty$;
\item
$I(a,a')=I(b,b')=1$, $I(a,b')=I(b,a')=\infty$;
\item
$p(a,a')=1$, $p(a,b')=p(b,b')=p(b,a')=\infty$;
\item
$q(b,b')=1$, $q(b,a')=q(a,a')=q(a,b')=\infty$;
\item
$u(a,b')=1$, $u(a,a')=u(b,a')=u(b,b')=\infty$;
\item
$u^*(b,a')=1$, $u^*(b,b')=u^*(a,b')=u^*(a,a')=\infty$;
\item
$s(a,b')=s(b,a')=1$, $s(a,a')=s(b,b')=\infty$.
\end{enumerate}

Note that $p,q$ are idempotents, $u$ and $u^*$ are partial isometries, $u^*u=p$, $uu^*=q$, and $s$ is a symmetry.
Let $L_0=\langle \delta_0\rangle$, $L_1=\langle\delta_p,\delta_{u^*}\rangle$, $L_2=\langle\delta_q,\delta_u\rangle$, $L_3=\langle\delta_I,\delta_s\rangle$. Then $V_0=L_0$, $V_p=V_u=L_1\oplus L_0$, $V_q=V_{u^*}=L_2\oplus L_0$, $V_I=V_s=L_0\oplus L_1\oplus L_2\oplus L_3$. 

We have $\lambda_d|_{L_0}=\id$ for any $d\in M(X)$, $\lambda_u(L_1)=L_2$, $\lambda_{u^*}(L_2)=L_1$ and $\lambda_s|_{L_1\oplus L_2}=\lambda_u|_{L_1\oplus L_2}+\lambda_{u^*}|_{L_1\oplus L_2}$, so $\lambda_u$, $\lambda_{u^*}$ and $\lambda_s$ restricted to the invariant subspace $L_1\oplus L_2$ generate the algebra isomorphic to $M_2(\mathbb C)$. Taking into account the invariant subspaces $L_0$ and $L_3$, where $M(X)$ acts by scalars, we get $C^*(M(X))\cong\mathbb C\oplus\mathbb C\oplus M_2(\mathbb C)\subset M_7(\mathbb C)$.

\end{example}

\section{Sufficient condition for an isomorphism $M(X)\cong M(Y)$}

Given two metric spaces, $X$ and $Y$, consider all metrics $d$ on the disjoint union $X\sqcup Y$ such that
\begin{itemize}
\item
$d|_X=d_X$, $d|_Y=d_Y$;
\item
$d(X,Y)\neq 0$.
\end{itemize}
Let $\mathcal M(X,Y)$ denote the set of all such metrics. 

Recall that, given a metric $d$ on $X\sqcup Y$, the Hausdorff distance between $X$ and $Y$ is $d_H(X,Y)=\max(\sup_{x\in X}d(x,Y),\sup_{y\in Y}d(y,X))$, and the Gromov--Hausdorff distance between $X$ and $Y$ is $\inf d_H(X,Y)$, where the infimum is over all metrics on $X\sqcup Y$ that equal $d_X$ and $d_Y$ on $X$ and $Y$, respectively. Note that the Gromov--Hausdorff distance may be (and often is) infinite.

\begin{lem}
The Gromov--Hausdorff distance between $X$ and $Y$ equals $\inf_{d\in \mathcal M(X,Y)}d_H(X,Y)$.

\end{lem}
\begin{proof}
If $d$ is a metric on $X\sqcup Y$ that equals $d_X$ and $d_Y$ on $X$ and $Y$ respectively, and $d_H(X,Y)=0$ then, for any $\varepsilon>0$, set $d^\varepsilon|_X=d_X$, $d^\varepsilon|_Y=d_Y$, and $d^\varepsilon(x,y)=d(x,y)+\varepsilon$ for $x\in X$, $y\in Y$. It is clear that $d^\varepsilon$ is a metric in $\mathcal M(X,Y)$ and $d^\varepsilon_H(X,Y)\geq\varepsilon$, so it suffices to take the infimum over metrics for which the distance between $X$ and $Y$ is non-zero.  

\end{proof}

\begin{prop}
Suppose that the Gromov--Hausdorff distance between $X$ and $Y$ is finite. Then $M(X)$ and $M(Y)$ are isomorphic. 

\end{prop}
\begin{proof}
By assumption, there exists $\rho\in\mathcal M(X,Y)$ and $C>0$ such that $\rho(x,Y)<C$ and $\rho(y,X)$ for any $x\in X$ and any $y\in Y$. Then $\rho^*\rho\in\mathcal M(X)$, and $\rho^*\rho(x,x')=\inf_{z\in X}2\rho(x,z')<2C$ for any $x\in X$, hence $\rho^*\rho\sim I$, where $I\in\mathcal M(X)$ is defined in Example \ref{I}. Similarly, $\rho\rho^*\sim I$ in $\mathcal M(Y)$. 

For $d\in\mathcal M(X)$, $b\in\mathcal M(Y)$, set $\varphi(d)=\rho d\rho^*\in\mathcal M(Y)$, $\psi(b)=\rho^* b\rho$.  Clearly, $\varphi$ and $\psi$ pass to maps $\bar\varphi:M(X)\to M(Y)$ and $\bar\psi:M(Y)\to M(X)$, respectively. These maps are semigroup homomorphisms, as $[\rho^*\rho]$ and $[\rho\rho^*]$ are the unit elements in $M(X)$ and in $M(Y)$, respectively.

Finally, $\psi\circ\varphi(d)=\rho^*\rho d\rho^*\rho\sim d$, hence $\bar\psi\circ\bar\varphi=\id_{M(X)}$. Similarly, $\bar\varphi\circ\bar\psi=\id_{M(Y)}$.

\end{proof}

\section{Subgroup of invertibles}

An element $[d]\in M(X)$ is invertible if $[d^*d]=[dd^*]=[I]$. It is clear that the invertible elements form a group. Here we describe this group.

\begin{defn}\label{Def-isom}
A map $f:X\to X$ is an almost isometry if 
\begin{itemize}
\item[(i1)]
there exists $C>0$ such that 
$$
d_X(x,\tilde x)-C\leq d_Y(f(x),f(\tilde x))\leq d_X(x,\tilde x)+C
$$
for any $x,\tilde x\in X$;
\item[(i2)]
there exist a map $g:X\to X$ and $D>0$ such that $d_X(g\circ f(x),x)<D$ and $d_X(f\circ g(x),x)<D$ for any $x\in X$.
\end{itemize}

\end{defn}

Note that if such a map $g$ exists then it automatically satisfies (i1), possibly with different $C$.

Any isometry is patently an almost isometry. Another example of an almost isometry for $X=\Gamma$, where $\Gamma$ is a finitely generated group with the word-length metric, is provided by conjugation by a fixed element $g\in\Gamma$.

Given an almost isometry $f:X\to X$, set
$$
d^f(x,y')=\inf_{z\in X}d_X(x,z)+C+d_X(f(z),y),\quad x,y\in X.
$$
It was shown in \cite{Manuilov-Morphisms} that $d^f$ is a metric (one has to check four triangle inequalities).

If $f,g:X\to X$ are almost isometries as in Definition \ref{Def-isom} then (i2) implies that $[d^fd^g]=[d^gd^f]=[I]$.

\begin{prop}
Let $d\in\mathcal M(X)$ and let $[d]\in M(X)$ be invertible. Then there exists an almost isometry $f$ of $X$ such that $[d^f]=[d]$. 

\end{prop}
\begin{proof}
If $[d]$ is invertible then $[d^*d]=[dd^*]=[I]$, so there exists $C>0$ such that $\inf_{z\in X}[d(x,z')+d(z',x)]<C$ and $\inf_{z\in X}[d(x',z)+d(z,x')]<C$ for any $x\in X$. Therefore there exist $u,v\in X$ such that $d(x,u')<C/2$ and $d(x',v)<C/2$.

Set $f(x)=u$, $g(x)=v$. Then $d(x,f(x)')<C/2$ and $d(x',g(x))<C/2$ for any $x\in X$, hence $d(f(x)',g(f(x)))<C/2$, and, by the triangle inequality, 
$$
d_X(x,g\circ f(x))\leq d(x,f(x)')+d(f(x)',g(f(x)))<C. 
$$
Similarly one gets $d_X(x,f\circ g(x))<C$. 

Let $x,\tilde x\in X$. Then, by the triangle inequality,
$$
d_X(f(x),f(\tilde x))\leq d(f(x),x')+d_X(x',\tilde x')+d(\tilde x',f(\tilde x))\leq d_X(x,\tilde x)+C
$$
and
$$
d_X(f(x),f(\tilde x))\geq -d(f(x),x')+d_X(x',\tilde x')-d(\tilde x',f(\tilde x))\geq d_X(x,\tilde x)-C,
$$
hence $f$ is an almost isometry.

It remains to check that $[d^f]=[d]$. Taking $z=x$ and using the triangle inequality, we get
\begin{eqnarray*}
d^f(x,y')&=&\inf_{z\in X}[d_X(x,z)+d_X(f(z),y)+C]\leq d_X(f(x),y)+C=d_X(f(x)',y')+C\\
&\leq&d(f(x)',x)+d(x,y')+C\leq d(x,y')+3C/2.
\end{eqnarray*}

To prove the estimate from below, note that
\begin{eqnarray*}
d_X(f(z),y)&\geq& d_X(g\circ f(z),g(y))-C/2\geq d_X(z,g(y))-d_X(g\circ f(z),z)-C/2\\
&\geq& d_X(z,g(y))-C-C/2=d_X(z,g(y))-3C/2,
\end{eqnarray*}
hence, by the triangle inequality,
\begin{eqnarray*}
d^f(x,y')&=& \inf_{z\in X}[d_X(x,z)+d_X(f(z),y)+C]\geq \inf_{z\in X}[d_X(x,z)+d_X(z,g(y))]-C/2\\
&\geq& d_X(x,g(y))-C/2\geq d(x,y')-d(y',g(y))-C-C/2\\
&\geq& d(x,y')-3C/2.
\end{eqnarray*}

\end{proof}

\section{Coarse version}

Two metrics, $d_1$, $d_2$, on $X$ are coarsely equivalent if there exists a monotonely increasing function $f$ on $[0,\infty)$ with 
$\lim_{t\to\infty}f(t)=\infty$ such that
$$
f^{-1}(d_2(x,y))\leq d_1(x,y)\leq f(d_2(x,y))
$$
for any $x,y\in X$. 

All our results hold also for the coarse equivalence classes of metrics on the double of $X$. This gives a smaller quotient inverse semigroup $M_c(X)$ of coarse equivalence classes. We may also use the fact that the image of an inverse semigroup, under a semigroup homomorphism, is an inverse semigroup.

\section*{}
\textbf{Acknowledgment.} The author expresses his gratitude to the referees for valuable comments and for suggested nicer versions of certain proofs.

\end{document}